\documentclass[11pt,a4paper,reqno]{amsart}
\usepackage{amssymb,tikz,tikz-3dplot}

\usepackage[hidelinks]{hyperref}
\usepackage[english]{babel}
\usepackage[top=3cm,right=3cm,left=3cm,bottom=3cm]{geometry}
\usetikzlibrary{arrows.meta}

\usepackage{float,wrapfig}
\restylefloat{figure}
\usepackage[labelfont=sc]{caption}

\allowdisplaybreaks[4]
\makeatletter
\@namedef{subjclassname@2020}{%
  \textup{2020} Mathematics Subject Classification}
\makeatother

\newtheorem{thm}{Theorem}
\newtheorem{lemma}[thm]{Lemma}
\newtheorem{prop}[thm]{Proposition}
\newtheorem{rem}[thm]{Remark}
\newtheorem{cor}[thm]{Corollary}
\newtheorem{df}[thm]{Definition}

\newtheorem{notations}[thm]{Notations}

\DeclareMathOperator{\End}{End(E)}
\DeclareMathOperator{\Ends}{End}
\DeclareMathOperator{\rk}{rk}
\DeclareMathOperator{\spec}{Spec}

\DeclareMathOperator{\Ker}{Ker}

\newcommand{\sgp}{S_{\Gi}^{\perp}}

\newcommand{\sgpe}{S_{\Gi (\ec)}^{\perp}}

\newcommand{\G}{\mathbb{G}}
\newcommand{\Gi}{\mathbb{G}_i}
\newcommand{\R}{\mathbb{R}}
\newcommand{\C}{\mathbb{C}}
\newcommand{\CP}{\mathbb{C}\mathrm{P}}

\renewcommand{\emptyset}{\varnothing}

\newcommand{\tr}{\mathrm{Tr}}
\newcommand{\inner}[1]{\left<#1\right>}
\newcommand{\hp}{\mathbb{HP}}
\newcommand{\hpu}{\mathbb{HP}_1}
\newcommand{\hpd}{\mathbb{HPD}}
\newcommand{\hpdu}{\mathbb{HPD}_1}
\newcommand{\Hb}{\mathbb{H}}
\newcommand{\Sc}{\ensuremath{\mathcal{S}}}
\newcommand{\ec}{\ensuremath{\mathcal{E}}}
\newcommand{\cc}{\ensuremath{\mathcal{C}}}
\newcommand{\pc}{\ensuremath{\mathcal{P}}}

\title{On the pseudo-manifold of quantum states}
\author[F.~D'Andrea and D.~Franco]{Francesco D'Andrea and Davide Franco}
\date{02/11/2020}

\address[F.~D'Andrea]{Dipartimento di Matematica e Applicazioni,
Universit\`a di Napoli ``Federico II'' and I.N.F.N. Sezione di Napoli, Complesso MSA, Via Cintia, 80126 Napoli, Italy}
\email{francesco.dandrea@unina.it}

\address[D.~Franco]{Dipartimento di Matematica e Applicazioni,
Universit\`a di Napoli ``Federico II'', Complesso MSA, Via Cintia, 80126 Napoli, Italy}
\email{davide.franco@unina.it}

 % Quantum state spaces, operational and probabilistic concepts


\begin{document}
\setlength{\leftmargini}{2.5em} %itemize/quote left margin

\begin{abstract}
There are various statements in the physics literature about the stratification of quantum states, for example into orbits of a unitary group, and about generalized differentiable structures on it. Our aim is to clarify and make precise some of these statements. For $A$ an arbitrary finite-dimensional C*-algebra and $U(A)$ the group of unitary elements of $A$, we observe that the partition of the state space $\mathcal{S}(A)$ into $U(A)$ orbits is not a decomposition and that the decomposition 
into orbit types is not a stratification (its pieces are not manifolds without boundary), while there is a natural Whitney stratification into matrices of fixed rank. For the latter, when $A$ is a full matrix algebra, we give an explicit description of the pseudo-manifold structure (the conical neighbourhood around any point). We then make some comments about the infinite-dimensional case.
\end{abstract}

\maketitle

\vspace*{-5mm}

\section{Introduction}
While classical mechanics can be formulated in terms of time evolution of physical systems living in a manifold and of functions defined on such a manifold, an algebraic formulation that works both in classical and quantum mechanics consists in studying the time evolution of states of a C*-algebra (whose self-adjoint elements we may interpret as the observables of our system). The pure state space of a C*-algebra replaces the phase space of classical mechanics, and 
more generally the state space replaces the space of probability distributions on a manifold, that is the starting point of statistical mechanics.
An interesting question is to investigate whether the state space $\mathcal{S}(A)$ of a C*-algebra $A$ admits a differentiable structure, in some generalized sense.

A common statement that is found in physics literature is that, when $A=M_n(\C)$ is the algebra of all $n\times n$ complex matrices, the set of all density matrices form a stratified space with strata given by $U(n)$ orbits (see e.g.~\cite{BZ} or \cite{CCIMV}). When $A$ is the algebra of all bounded operators on a separable Hilbert space one has a similar decomposition, with orbits that turn out to be smooth Banach manifolds \cite{CIJM}.
Another statement that can be found e.g.~in \cite{GKM} is that density matrices are stratified by rank. Indeed, it is observed already in \cite{DR92} that positive matrices --- what they call ``non-normalized density matrices'' --- are stratified by rank.
It is clear from these statements that the word ``stratification'' is being used with several different meanings, since for example the one by rank is a stratification in the mathematical sense of the word (cf.~Section \ref{sec:revB}), while the partition into orbits is clearly not locally finite.

The aim of this paper is to show that the state space of an arbitrary finite-dimensional C*-algebra is indeed a ``good'' singular space: it is a stratified space in the sense of Whitney whose structure can be described rather explicitly. The bulk of the paper is the explicit description of the conical neighbourhood around any point.
The main reason for focusing on stratified spaces in the sense of Whitney is that it is the class of singular spaces that are as close as possible to smooth manifolds. When working with such spaces, many issues can be  be addressed by working out ``one stratum at a time'', reducing many problems to questions about smooth manifolds. In addition, one has a cohomological theory and a Morse theory with properties very similar to those for differentiable manifolds \cite{GMP}.

\smallskip

Since we believe that the topic is of major interest in physics, we start with a short review of some background material, that we collect here to make the paper accessible to a more general audience (as well as to fix the notations). 
In Section \ref{sec:revA} we recall some properties of state spaces, while in Section \ref{sec:revB} we recall the notions of decomposed space, stratified space (in the sense of Whitney) and pseudo-manifold. These sections can be readily skipped by the experts of the field.
In Section \ref{sec:prel} we present some basic lemmas about the topology of state spaces: in particular, for a direct sum $A=A_1\oplus A_2$ of C*-algebras we explain how a stratification of $\mathcal{S}(A_1)$ and $\mathcal{S}(A_2)$ induces one of $\mathcal{S}(A)$. In Section \ref{sec:orbits} we discuss the decomposition of states into orbit types and observe, through simple examples, that is not a stratification.
Section \ref{sec:MnC} is devoted to the study of the local structure of the state space of a full matrix algebra.
Section \ref{sec:infdim} contains a short discussion about the infinite-dimensional case.

\subsection{States of C*-algebras}\label{sec:revA}
Here by a C*-algebra we shall always mean a complex one. We focus on \emph{unital} C*-algebras, the algebraic counterpart of \emph{compact} Hausdorff spaces.

Let $A$ be a unital C*-algebra. The dual space $A^*$ of bounded linear functionals \mbox{$f:A\to\C$} is a Banach space with norm
$$
\|f\|:=\sup_{a\in A\smallsetminus\{0\}}\frac{|f(a)|}{\|a\|} \;.
$$
The unit ball $B:=\{\varphi\in A^*:\|\varphi\|=1\}$ is a convex subset of $A^*$, and by Banach-Alaoglu theorem it is
compact in the weak-* topology: this is the topology we will consider on $A^*$ unless stated otherwise.
Both the norm and weak-* topology are Hausdorff, and they coincide if and only if $A$ is finite-dimensional (see e.g.~\cite[Cap.~5]{Kes}).

A functional $\varphi\in A^*$ is called positive if $\varphi(a^*a)\geq 0\;\forall\;a\in A$, and a necessary and sufficient condition is $\varphi(1)=\|\varphi\|$ \cite[Cor.~3.3.4]{Mur}. The \emph{state space} $\mathcal{S}(A)$, i.e.~the set of positive linear functionals with norm $1$, is then the intersection of the unit ball $B$ with the hyperplane $\{\varphi\in A^*:\varphi(1)=1\}$. It is a compact (hence closed) convex subset of $B$.
Extremal points (states that cannot be written as convex combinations of other states) are called \emph{pure states}.
The set of pure states, which we denote by $\mathcal{P}(A)$, is in general not closed (hence not compact) in the weak-* topology, cf.~\cite[Pag.~261]{Kad}; its closure $\overline{\mathcal{P}(A)}$ is the so-called \emph{pure state space} of $A$ (and it may contain states that are not pure).
By Krein-Milman theorem, %\footnote{Stating that if $K$ a compact convex subset of a locally convex Hausdorff topological vector space, then $K$ is the closed convex hull of its extreme points.} 
$\mathcal{S}(A)$ is the closure of the convex hull of $\mathcal{P}(A)$ (in particular $\mathcal{P}(A)$ is not empty).

If $A=\mathcal{B}(H)$ is the C*-algebra of all bounded operators on a Hilbert space $H$, the set $\mathcal{P}(A)$ is closed in the weak-* topology if and only if $H$ is finite-dimensional \cite[Ex.~4.6.69]{Kad}. More generally, $\mathcal{P}(A)$ is closed for every
finite-dimensional C*-algebra $A$ (cf.~Cor.~\ref{cor:puress}).

By Riesz-Markov-Kakutani representation theorem, if $A=C(X)$ is the C*-algebra of continuous functions on a compact Hausdorff space $X$ (with sup norm), every state $\varphi$ is given by a unique regular Baire probability measure $\mu$ on $X$ via the formula:
$$
\varphi(f)=\int_Xf(x)d\mu(x) \;.
$$
Pure states are normalized point measures, i.e.~evaluations at points of $X$. The map
\begin{equation}\label{eq:evaluation}
X\to\mathcal{P}(C(X)) \;,\qquad x\mapsto\mathrm{ev}_x,
\end{equation}
is a homeomorphism between $X$ with its topology and $\mathcal{P}(C(X))$ with weak-* topology (see e.g.~\cite[Sect.~1.3]{GVF}). In particular $\mathcal{P}(C(X))$ is compact, hence closed in $\mathcal{S}(C(X))$.

\begin{rem}
Because of the homeomorphism \eqref{eq:evaluation}, if $X$ is a compact smooth manifold, there is a smooth structure on $\mathcal{P}(C(X))$ compatible with the weak-* topology.
\end{rem}

A natural question is whether one can equip the state space of an arbitrary C*-algebras with a generalized smooth structures compatible with the weak-* topology and that possibly, when $A=C(X)$ are functions on a manifold, allow us to reconstruct $X$ from the pure state space through the map \eqref{eq:evaluation}.

An analogous question in the framework of metric structures --- the study of metrics (distances) on state spaces inducing the weak-* topology --- led to a well developed line of research, due mainly to M.A.~Rieffel and his school (see e.g.~the seminal paper \cite{Rie} or the recent review \cite{Lat}). A celebrated extended metric coming from a generalized Dirac operators is the spectral distance of A.~Connes \cite{Con94}, that generalizes the Wasserstein distance of order $1$ of transport theory to the non-commutative world (see e.g.~in \cite{DM10}). It is also worth mentioning that, if $X$ is a connected compact Riemannian manifold (with no boundary) and $\omega_g$ the Riemannian density, one can consider the space $\mathcal{P}^\infty(X)$ of probability measures of the form $f\omega_g$ with $f\in C^\infty(X)$; such a space, equipped with the Wasserstein metric of order $2$, is a formal Riemannian manifold. In particular, it is an infinite-dimensional smooth manifold in a suitable sense (see \cite{Lott} for the details), with underlying topology given by the smooth topology.

\medskip

In this paper we will be mainly concerned with finite-dimensional C*-algebras. By the well-known classification theorem (see for example \cite[Theorem III.1.1]{Dav96}), any such algebra is isometrically \mbox{$*$-isomorphic} to a finite direct sum $A$ of full matrix algebras:
\begin{equation}\label{eq:finitedim}
A=\bigoplus_{i=1}^kM_{n_i}(\C) \;,
\end{equation}
where $n_1,\ldots,n_k$ are (not necessarily distinct) positive integers. %and $M_j(\C)$ denotes the algebra of $j\times j$ complex matrices. 
Let $n:=n_1+\ldots+n_k$. We will think of $A$ in \eqref{eq:finitedim} as the subalgebra of $M_n(\C)$ made of all block diagonal matrices with blocks of size $n_1,\ldots,n_k$. Let:
$$
\inner{a,b}:=\tr(a^*b) \;,\qquad\forall\;a,b\in M_n(\C),
$$
be the Hilbert-Schmidt inner product. Let us denote by $\mathbb{D}(A)$ the set of all $\rho\in A$ that are positive and with trace $1$ -- we call such elements \emph{density matrices} -- and by $\mathbb{P}(A)\subset\mathbb{D}(A)$ the set of Hermitian matrices $\rho\in A$ satisfying $\rho^2=\rho$ -- we call such elements \emph{projections}. The isomorphism of vector spaces
\begin{equation}\label{eq:duality}
A\to A^* \;,\qquad \rho\mapsto\inner{\rho,\,.\,},
\end{equation}
gives a bijection between $\mathbb{D}(A)$ and $\mathcal{S}(A)$. Every density matrix $\rho$ is a convex combination of its eigenprojections: the corresponding state is pure if and only if $\rho$ is itself a projection.
That is, the map \eqref{eq:duality} restricts to a bijection $\mathbb{P}(A)\to\mathcal{P}(A)$.

As costumary, the set of selfadjoint elements of $A$ will be denoted by $A^{\mathrm{sa}}$, and the set of positive elements by $A_+$. The former is a real vector subspace of $A$, and the latter is a convex cone in $A^{\mathrm{sa}}$.

\subsection{Stratified spaces}\label{sec:revB}
We will adopt the terminology of \cite{GMP} (see also \cite{Pfl}). In the following, by ``smooth manifold'' we mean a finite-dimensional Hausdorff second countable $C^\infty$ real manifold without boundary.

\begin{df}\label{def:decomposition}
Let $I$ be a set. An \emph{$I$-decomposition} of a topological space $X$ is a locally finite collection of disjoint locally closed subsets $S_i\subset X$, one for each $i\in I$, such that
\begin{itemize}\itemsep=2pt
\item[(1)] $X=\bigcup_{i\in I}S_i$,
\item[(2)] $S_i\cap\overline{S}_j\neq\emptyset$ implies $S_i\subset\overline{S}_j$.
\end{itemize}
The sets $\{S_i\}$ are called \emph{pieces} of the decomposition, and a topological space $X$ equipped with an $I$-decomposition will be called an \emph{$I$-decomposed space}.
\end{df}

By ``locally finite'' we mean that every point has a neighborhood which intersects only finitely many pieces; a set $S\subset X$ is ``locally closed'' if it is the intersection of an open and a closed subset of $X$.
Condition (2) is the so-called \emph{frontier condition}, saying that if a piece $S_i$ intersects the closure of $S_j$, it must lie in the frontier $\overline{S}_j\smallsetminus S_j$.
It follows easily from the frontier condition that the relation
$$
i\leq j \iff S_i\subset\overline{S}_j
$$
defines a partial order on $I$, and that for all $j\in I$:
\begin{equation}\label{eq:shortproof}
\overline{S}_j=\bigcup_{i\leq j}S_i \;.
\end{equation}
As a consequence:

\begin{rem}\label{rem:closed}
$S_i$ is closed $\iff$ $i$ is a minimal element of $(I,\leq)$.\footnote{If $S_i=\overline{S}_i$, from $\overline{S}_i\cap S_j=S_i\cap S_j=\emptyset\;\forall\;j\neq i$ it follows that $\nexists\;j< i$.
Conversely, if $\nexists\;j< i$, from \eqref{eq:shortproof} it follows that $\overline{S}_i=S_i$.}
\end{rem}

On any subset $X$ of an $I$-decomposed space $Y=\bigcup_{i\in I}R_i$ there is an obvious $I$-decomposition with pieces $S_i:=R_i\cap X$ and compatible with the subspace topology.

An example is the decomposition of $M_n(\C)$ into subsets of matrices of fixed rank. It follows that any subset of $M_n(\C)$ is a decomposed space, in particular the set of density matrices of a finite-dimensional C*-algebra is a decomposed space. In this example, $I$ is totally ordered and the least element, the unique closed piece of the decomposition, is given by rank $1$ density matrices, i.e.~pure states.

\begin{df}\label{def:stratification}
Let $X$ be a closed subset of a smooth manifold $M$.
An $I$-decomposition of $X$ is called a \emph{Whitney stratification} provided each $S_i$ is a smooth embedded submanifold of $M$ and, for all $i\leq j$, Whitney's condition (B) is satisfied:
\begin{itemize}
\item[(B)] suppose $(x_k)$ is a sequence of points of $S_j$ and $(y_k)$ a sequence of points of $S_i$, both converging to some $y\in S_i$, and let $\ell_k$ be the line through $x_k$ and $y_k$ in some local chart on $M$.\footnote{For $k$ big enough, all points are contained in a chart around $y$.} If $(\ell_k)$ converges to some limiting line $\ell$ and the tangent spaces $T_{x_k}S_j$ converge to some limiting space $\tau$ in the Grassmann bundle of $TM$, then $\ell\subset\tau$.
\end{itemize}
We will say that $X$ is a \emph{(Whitney) stratified space}, and call the sets $\{S_i\}$ \emph{strata} of the stratification.
\end{df}

In the above definition, Whitney's condition (A) is omitted, since it is implied by (B). If Whitney's condition (B) is satisfied in a chart around $y$, it is satisfied in any chart around $y$ \cite[Lemma 1.4.4]{Pfl}.
In \cite{Pfl}, there is an additional assumption in the definition of decomposed space: $X$ is required to be a paracompact separable Hausdorff space, and each $S_i$ a smooth manifold in the induced topology. Such conditions are automatically satisfied if $X$ is a Whitney stratified space.

For $k\geq 0$, we call \emph{$k$-skeleton} of a Whitney stratified space $X$ the decomposed space:
$$
X^k:=\bigcup_{i\in I,\dim(S_i)\leq k}S_i
$$
Clearly $X^k\subset X^{k+1}$. The least $n$ for which $X=X^n$ is the \emph{dimension} of the stratification.

Notice that any Whitney stratification of the sphere $\mathbb{S}^{n-1}\subset\R^n$ defines a Whitney stratification of $\R^n$, called \emph{conical stratification} induced by the one of the sphere, whose strata are the open cones of the strata of $\mathbb{S}^{n-1}$ plus the origin.

An important result about the local structure of stratifications is that, if $X$ is an \mbox{$n$-dimensional} Whitney stratified space, for any point $x$ of a $k$-dimensional stratum there exists a Whitney stratification of $\mathbb{S}^{n-k-1}$ and a homeomorphism taking strata into strata from a neighbourhood $N_x$ of $x$ to $\R^k\times\R^{n-k}$,
where $\R^{n-k}$ is given the conical stratification and $\R^k$ the one with only one stratum.

This motivates the following recursive definition.

\begin{df}\label{def:pseudoM}
A $0$-dimensional \emph{pseudo-manifold} is a countable set with discrete topology. For $n>0$, an $n$-dimensional pseudo-manifold is a paracompact Hausdorff topological space $X$ with a filtration
$$
X=X^n\supset X^{n-1}\supset\ldots\supset X^1\supset X^0
$$
by closed subsets such that:
\begin{itemize}\itemsep=0pt
\item[(1)] $X\smallsetminus X^{n-1}$ is dense in $X$,
\item[(2)] any $x\in X^i\smallsetminus X^{i-1}$ admits a neighbourhood $N_x$ in $X$ of the form
$$
N_x\stackrel{\phi}{\simeq}\R^i\times\mathcal{C}(L) ,
$$
where $L=L^{n-i-1}\supset\ldots\supset L^1\supset L^0$ is an $(n-i-1)$-dimensional pseudo-manifold, $\mathcal{C}(L)=L\times [0,1)/L\times\{0\}$ the cone over $L$, and $\phi$ a homeomorphism preserving the filtration:
$$
N_x\cap X^{i+j+1}\stackrel{\phi}{\simeq}\R^i\times\mathcal{C}(L_j)
$$
for all $0\leq j\leq n-i-1$.\footnote{Without the property (1), $X$ is what is called a ``topologically stratified spaces'' in \cite{Kir};
in the definition of ``pseudo-manifold'' of dimension $n$ in \cite{Kir} they make the additional assumption that $X^{n-1}=X^{n-2}$.}
\end{itemize}
\end{df}

Every Whitney stratified space is a pseudo-manifold, with filtration given by skeletons.

\smallskip

We close this section with a theorem that will be useful later on. Given a Lie group $G$ with a proper smooth action on a smooth manifold $M$, we say that $x,y\in M$ are of the same \emph{orbit type} if the isotropy groups of the two points are conjugated subgroups of $G$. If $H\subset G$ is a subgroup, we denote by $(H)$ its conjugacy class and by $M_{(H)}$ be the set of points of $M$ with isotropy group conjugated to $H$.
If $I$ is the set of conjugacy classes of subgroups $H$ of $G$, one has a partition:
$$
M=\bigcup_{(H)\in I}M_{(H)} \;.
$$
Such a partition is in fact a Whitney stratification, with partial order given by:
\begin{equation}\label{eq:partialorder}
(K)\leq (H)\iff H\text{ is conjugated (in $G$) to a subgroup of }K
\end{equation}
(notice the counter-intuitive order). The strata of this stratification are called \emph{orbit types}.

\begin{thm}[{\protect\cite[Theorem 4.3.7]{Pfl}}]\label{thm:1}
Any proper smooth action of a Lie group $G$ on a smooth manifold $M$ induces a Whitney stratification, with strata given by orbit types and partial order \eqref{eq:partialorder}.
\end{thm}

If $M$ is a finite-dimensional real vector space with a linear action by a compact group $G$, one can prove that there is only a finite number of non-empty orbit types \cite[Prop.~4.2.8]{Sni13}.

%In a number of interesting cases one has only finitely many orbit types. For example when $M$ is compact \cite[Lemma 4.3.6]{Pfl} or when $M$ is a finite-dimensional real vector space with a linear action by a compact group $G$ \cite[Prop.~4.2.8]{Sni13}. We are interested both in compact subsets and in open submanifolds of finite-dimensional vector spaces.

\section{On the topology of quantum states}\label{sec:prel}

Recall that a state (of a finite-dimensional C*-algebra) is called \emph{maximally mixed} if its density matrix is proportional to the identity matrix.

As a preliminary remark, it is not difficult to prove that if $A$ is finite-dimensional, $\mathcal{S}(A)$ is homeomorphic to a closed ball. For $A$ as in \eqref{eq:finitedim}, let $H$ be the affine subspace of $A^{\mathrm{sa}}$ of all matrices with trace $1$. Positive invertible matrices form an open subset of $H$, and $\mathbb{D}(A)$ is a compact convex subset with non-empty interior (the maximally mixed state is, for example, an interior point).
Any compact convex subset of a finite-dimensional real vector space with non-empty interior is homeomorphic to a closed ball, and the homeomorphism sends extreme points to the boundary sphere (see e.g.~\cite[Prop.~5.1]{Lee}). Thus:

\begin{rem}
For $A$ as in \eqref{eq:finitedim}, $\mathcal{S}(A)$ is homeomorphic to the closed unit ball in $\R^m$, with $m:=\dim_{\C}A-1$, with a homeomorphism that sends $\mathcal{P}(A)$ to the unit sphere.
\end{rem}

This homeomorphism, however, is not a convex function and forgets most of the structure of the state space. For example, $\mathcal{S}(M_2(\C))$ and $\mathcal{S}(\C^4)$ are both homeomorphic to a 3-ball, but in the former case the pure state space is the boundary 2-sphere while in the second case it consists of 4 points (it follows from Cor.~\ref{cor:puress} that the boundary of the ball is not the pure state space, except in the simple cases $A=\C$ or $A=M_2(\C)$).

The next proposition characterizes the state space -- in particular the set of pure states -- of a direct sum of C*-algebras in terms of the state space of each factor. In the proposition, we \emph{don't} assume that the summands are finite-dimensional.
The topology considered on the state spaces is the weak-* one.

\begin{lemma}\label{lemma:directsum}
Let $A=\bigoplus_{i=1}^k A_i$ be a direct sum of finitely many unital C*-algebras $A_1,\ldots,A_k$ and, for every $i=1,\ldots,k$, identify $A_i$ with its image in $A$ under the (non-unital) $*$-homomorphism
$$
a\mapsto\underbrace{0\oplus\ldots\oplus 0}_{i-1\textup{ times}}{}\oplus a\oplus \underbrace{0\oplus\ldots\oplus 0}_{k-i\textup{ times}} \;,\qquad\forall\;a\in A_i.
$$
Use the pullback of the projection onto the $i$-th summand to identify linear functionals on $A_i$ with (suitable) linear functionals on $A$.
Then
\begin{itemize}\itemsep=2pt
\item[(i)] Any $\varphi\in\mathcal{S}(A)$ is a convex combination $\varphi=\sum_{i=1}^k\lambda_i\varphi_i$ where $\varphi_i\in \mathcal{S}(A_i)$.

\item[(ii)] $\mathcal{P}(A)$ is homeomorphic to $\bigsqcup_{i=1}^k \mathcal{P}(A_i)$.
\end{itemize}
\end{lemma}

\begin{proof}
It is enough to prove the statement when $A=A_1\oplus A_2$ and then work by induction.
Let $\varphi\in \mathcal{S}(A)$  and assume that $\lambda_1:=\varphi(1\oplus 0)$ and $\lambda_2:=\varphi(0\oplus 1)$ are both non-zero. Then $\varphi_1(a\oplus b):=\varphi(a\oplus 0)/\lambda_1$ and $\varphi_2(a\oplus b):=\varphi(0\oplus b)/\lambda_2$ are states on $A$ and $\varphi=\lambda_1\varphi_1+\lambda_2\varphi_2$. Such a $\varphi$ is not pure.

For any positive linear functional $\psi$ on a unital C*-algebra $B$, since $\|\psi\|=\psi(1)$, if $\psi(1)=0$ then $\psi$ is identically zero. Thus in the above notations if, for example, $\lambda_1=0$ then $\varphi=\varphi_2$. Either way, the decomposition (i) is valid.

From the above discussion, any $\varphi\in \mathcal{P}(A)$ either vanishes on $A_1\oplus 0$ or on $0\oplus A_2$. In the former case $a\mapsto\varphi(a\oplus 0)$ gives a pure state of $A_1$, in the latter $a\mapsto\varphi(0\oplus a)$ gives a pure states of $A_2$. There is then a bijection $\mathcal{P}(A)\to\mathcal{P}(A_1)\sqcup\mathcal{P}(A_2)$, and we want to prove that it is a homeomorphism.\footnote{At a set-theoretic level, a proof of this bijection is already in \cite[Lemma 2.1]{Mar}.}

Any convergent net of states vanishing on $0\oplus A_2$ clearly converges to a state vanishing on $0\oplus A_2$. Thus $\mathcal{S}(A_1)$ is closed in $\mathcal{S}(A)$. Similarly $\mathcal{S}(A_2)$ is closed in $\mathcal{S}(A)$.
This proves that $\mathcal{P}(A)\simeq \mathcal{P}(A_1)\sqcup \mathcal{P}(A_2)$ with the standard topology on a disjoint union.
\end{proof}

As a corollary, one has an explicit description of the pure state space of arbitrary finite-dimensional C*-algebras.
Recall that there is a homeomorphism
$$
\CP^{n-1}\to\mathbb{P}(M_n(\C))
$$
(a diffeomorphism if we think of the latter as an embedded submanifold of $M_n(\C)$) that sends a point \mbox{$[z_1:\ldots:z_n]$} into the matrix $\rho=(z_i\overline{z}_j)$. Recall also that $\mathcal{P}(M_n(\C))$ is closed in $\mathcal{S}(M_n(\C))$  \cite[Ex.~4.6.69]{Kad}. Thus:

\begin{cor}\label{cor:puress}
For $A$ as in \eqref{eq:finitedim}, $\mathcal{P}(A)$ is closed in $\mathcal{S}(A)$ and homeomorphic to a disjoint union 
of complex projective spaces: $\mathcal{P}(A)\simeq\bigsqcup_{i=1}^k\CP^{n_i-1}$.
\end{cor}

\begin{wrapfigure}[6]{r}{5.3cm}
\colorlet{mygreen}{green!50!black}
\colorlet{myred}{red!80!black}
\tdplotsetmaincoords{70}{70}

\begin{center}
\begin{tikzpicture}[tdplot_main_coords,scale=1.2,font=\small]

\coordinate (A) at (-1.5,0,1);
\coordinate (B) at (1.5,0,1);
\coordinate (C) at (0,3,0);
\coordinate (D) at (0,3,2);

\path[semithick]
		(A) edge[color=mygreen] node[below left] {$X$} (B)
		(C) edge[color=myred] node[right] {$Y$} (D)
		(A) edge[bend left=10] (D)
		(B) edge[bend right=10] (D)
		(B) edge[bend right=10] (C)
		(A) edge[dashed,bend left=10] (C);

\end{tikzpicture}
\end{center}

\vspace{-7pt}
\caption{Join $X\star Y$.}

\end{wrapfigure}
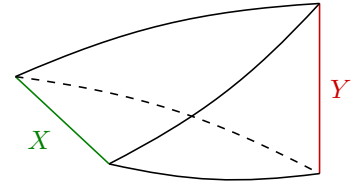
Recall that, given two topological spaces $X$ and $Y$, the \emph{join} $X\star Y$ is by definition the quotient of the product $X\times Y\times [0,1]$ by the equivalence relation generated by:
\begin{center}
$(x_0,y_1,0)\sim (x_0,y_2,0)$
\quad and \quad
$(x_1,y_0,1)\sim (x_2,y_0,1)$
\end{center}
for all $x_0,x_1,x_2\in X$ and all $y_0,y_1,y_2\in Y$.
Another consequence of Lemma \ref{lemma:directsum} is the following.

\begin{prop}
If $A=A_1\oplus A_2$ is a direct sum of two unital C*-algebras, then the convex map
$F:\mathcal{S}(A_1)\times\mathcal{S}(A_2)\times [0,1]\to \mathcal{S}(A)$ given by
\begin{equation}\label{eq:convexmap}
F(\varphi_1,\varphi_2,t):=(1-t)\varphi_1+t\varphi_2 \;,
\end{equation}
induces a homeomorphism between $\mathcal{S}(A)$ and $\mathcal{S}(A_1)\star\mathcal{S}(A_2)$.
\end{prop}

\begin{proof}
Consider the diagram
\begin{equation}\label{eq:quozsmooth}
\begin{tikzpicture}[baseline=(current  bounding  box.center)]

\node (A) at (0,2) {$\mathcal{S}(A_1)\times\mathcal{S}(A_2)\times[0,1]$};
\node (B) at (4,2) {$\mathcal{S}(A)$};
\node (C) at (2,0) {$\mathcal{S}(A_1)\star\mathcal{S}(A_2) $};

\path[font=\footnotesize,-To]
      (A) edge node[above] {$F$} (B)
      (C) edge[dashed] node[right=3pt]{$\widetilde{F}$} (B)
      (A) edge node[left=3pt,pos=0.53]{$\pi$} (C);

\end{tikzpicture}
\end{equation}
where $\pi$ is the quotient map and $F$ is the map \eqref{eq:convexmap}.
As before, here $\varphi_1(a\oplus b):=\varphi_1(a)$ and $\varphi_2(a\oplus b):=\varphi_2(b)$ for all $a\in A_1$ and $b\in A_2$.
It follows from Lemma \ref{lemma:directsum} that $F$ is surjective.

Suppose
$$
F(\varphi_1,\varphi_2,t)=F(\psi_1,\psi_2,s) \;.
$$
By applying both sides to elements of the form $a\oplus 0$ and $0\oplus b$ we see that the above condition is equivalent to
$$
(1-t)\varphi_1=(1-s)\psi_1 \qquad\text{and}\qquad t\varphi_2=s\psi_2 \;.
$$
Applying both sides to the unit element of $A_1$ risp.~$A_2$ we get $t=s$. If $0<t=s<1$ we get $(\varphi_1,\varphi_2)=(\psi_1,\psi_2)$.
If $t=s=0$ we get $\varphi_1=\psi_1$, while $\varphi_2,\psi_2$ are arbitrary.
If $t=s=1$ we get $\varphi_2=\psi_2$, while $\varphi_1,\psi_1$ are arbitrary.

The fibres of $F$ are then exactly the equivalence classes of the relation generated by $(\varphi_1,\varphi_2,0)\sim (\varphi_1,\psi_2,0)$ and $(\varphi_1,\varphi_2,1)\sim (\psi_1,\varphi_2,1)$. We deduce that $\widetilde{F}$ in \eqref{eq:quozsmooth} exists, it is a bijection, and it is continuous by the characteristic property of quotient maps.
Since the domain of $\widetilde{F}$ is compact, $\widetilde{F}$ is a homeomorphism.
\end{proof}

As costumary, let us identify $X\star Y$ as a set with $X\sqcup (X\times Y\times I)\sqcup Y$ where $I:= (0,1)$ is the open interval. It is then interesting to observe that any decomposition of $X$ and $Y$ (in the sense of Def.~\ref{def:decomposition}) induces a decomposition of $X\star Y$. Since we are interested in state spaces of unital C*-algebras, we will work in the category of compact Hausdorff spaces.

\begin{prop}\label{prop:11}
Let $X=\bigcup_{i\in I}R_i$ and $Y=\bigcup_{j\in J}S_j$ be decompositions of compact Hausdorff spaces. Then, a decomposition of the join is given by:
\begin{equation}\label{eq:decjoin}
X\star Y=\bigcup\nolimits_{i,j} R_i\sqcup ( R_i\times S_j\times I ) \sqcup S_j \;.
\end{equation}
\end{prop}

\begin{proof}
Let $\pi:X\times Y\times [0,1]\to X\star Y$ be the quotient map, and consider the disjoint sets $\mathcal{R}_i:=\pi(R_i\times Y\times\{0\})$, $\mathcal{S}_j:=\pi(X\times S_j\times\{1\})$ and $\mathcal{Q}_{ij}:=\pi(R_i\times S_j\times I)$.
We want to prove that the partition $X\star Y=\bigcup_i\mathcal{R}_i\cup\bigcup_j\mathcal{S}_j\cup\bigcup_{i,j}\mathcal{Q}_{ij}$ is a decomposition.

Notice that $\pi$ is a closed map, since the domain is compact and the codomain is Hausdorff.
Let $A\subset X$, $B\subset Y$, $C\subset X\times Y$.
Sets of the form $A\times Y\times\{0\}$, $X\times B\times\{1\}$ and $C\times I$ are saturated. For every $i$ and $j$,
by assumption the sets $R_i$ and $S_j$ are locally closed, that means
$$
R_i=A_i\cap \overline{R}_i \qquad\text{and}\qquad S_j=B_j\cap\overline{S}_j
$$
for suitable open sets $A_i\subset X$ and $B_j\subset Y$. It follows that:
\begin{align*}
\mathcal{R}_i &=\pi(A_i\times Y\times\{0\})\cap\pi(\overline{R}_i\times Y\times\{0\}) \;, \\
\mathcal{S}_j &=\pi(X\times B_j\times\{1\})\cap\pi(X\times \overline{S}_j\times\{1\}) \;, \\
\mathcal{Q}_{ij} &=\pi(A_i\times B_j\times I)\cap \pi(\overline{R}_i\times\overline{S}_j\times [0,1]) \;.
\end{align*}
Each one of the pieces $\mathcal{R}_i$, $\mathcal{S}_j$ and $\mathcal{Q}_{ij}$ is the intersection of the image of a saturated open set (hence open) and a closed set (hence closed, since $\pi$ is closed).

Since $\pi$ is closed, $\pi(\overline{D})=\overline{\pi(D)}$ for every subset $D$ of the domain. With this observation, the frontier condition is now a simple check.
\end{proof}

\section{On the decomposition into orbit types and by rank}\label{sec:orbits}

Let $A$ be a finite dimensional C*-algebra and consider the adjoint action $U(A)\times A\to A$, $(u,a)\mapsto uau^*$, of the group $U(A)$ of unitary elements of $A$. By invariance of the Hilbert-Schmidt inner product, the isomorphism \eqref{eq:duality} intertwines this action with its dual action on $A^*$. We are in the best possible scenario: a linear action of a compact Lie group on a finite-dimensional real vector space.
According to Theorem \ref{thm:1}, $A$ (or $A^*$ under the equivariant identification \eqref{eq:duality}) has a Whitney stratification into finitely many orbit types.

The subset $\mathbb{D}(A)$ of density matrices is $U(A)$-invariant: it can be partitioned into $U(A)$ orbits, that are compact smooth embedded submanifolds of $A$ (this is the point of view e.g.~of \cite{CCIMV}).
They have a decomposition into orbit types as well, in the sense of Def.~\ref{def:decomposition}, since as recalled on any subset of a decomposed space there is a natural decomposition obtained by intersecting the pieces with the subset. We will see however that orbit types in $\mathbb{D}(A)$ are not manifolds, but rather manifolds with boundary. This section is devoted to study some simple examples.
Before that, let us argue that, even if the one by orbit types is not a stratification, a stratification on $\mathbb{D}(A)$ exists.

\smallskip

Recall that a subset of $\R^n$ is called \emph{semialgebraic} if it is a finite union of subsets of $\R^n$ each one defined by finitely many polynomial equations and/or inequalities. Every semialgebraic set has a Whitney stratification with finitely many semialgebraic strata (see, e.g., \cite[Lemma I.2.10]{Shi97}).

By Sylvester's criterion, a Hermitian matrix is positive-semidefinite if and only if all of its principal minors are non-negative \cite[Cor.~7.1.5 and Theorem 7.2.5]{Hor}. Thus, $\mathbb{D}(A)$ is the subset of the real vector space $A^{\mathrm{sa}}$ defined by some polynomial inequalities (non-negativity of the principal minors) and a linear equation (trace equal to $1$). As such:

\begin{rem}
The set of density matrices $\mathbb{D}(A)$ is a closed semialgebraic subset of $A^{\mathrm{sa}}$. In particular, $\mathbb{D}(A)$ admits a semialgebraic Whitney stratification.
\end{rem}

An explicit description of such a stratification for full matrix algebras will be given in the next section. We will see that, for every $r\geq 1$, a stratum
$S_r$ is given by all the matrices of fixed rank $r$. The pure state space coincides with $S_1$, the unique stratum that is closed (see also Remark \ref{rem:closed}).

Let us now compare the decomposition of $\mathbb{D}(A)$ into orbit types and into subsets of matrices with fixed rank in few simple examples.

\subsection{A closed ball}
Let $A=M_2(\C)$. With the parametrization
\begin{equation}\label{eq:2x2}
\frac{1}{2}
\begin{bmatrix}
t+x_3 & x_1-\mathrm{i}x_2 \\
x_1+\mathrm{i}x_2 & t-x_3
\end{bmatrix} \;,\qquad t\in\R, x=(x_1,x_2,x_3)\in\R^3,
\end{equation}
of Hermitian matrices, $A_+$ is the cone of equation $\|x\|\leq t$, and $\mathbb{D}(A)$ is the intersection of such a cone with the hyperplane $t=1$, that is: a closed unit ball in $\R^3$.
Points in the boundary sphere --- usually referred to as \emph{Bloch sphere}, describing the pure states of a spin $1/2$ particle, see e.g.~\cite{BZ} --- correspond to pure states, points in the interior are mixed states. The origin of $\R^3$ corresponds to the maximally mixed state.

Under the homeomorphism between $\mathbb{D}(A)$ and the closed unit ball in $\R^3$, the adjoint action of $U(A)=U(2)$ becomes the natural action by rotations around the origin. The origin is a point of orbit type $(U(2))$, its complement is the dense orbit type $(U(1))$. We may observe already in this simple example that every non-empty open set intersects infinitely many orbits (the partition into orbits is not locally finite), and that the orbit type $(U(1))$ is not smooth submanifold (without boundary) of $\R^3$.

For $t=1$, the matrix \eqref{eq:2x2} has eigenvalues $\frac{1}{2}(1\pm\|x\|)$. Density matrices with rank $1$ are exactly those with $\|x\|=1$: the boundary of the ball. Density matrices with rank $2$ correspond to points in the interior of the ball.

\subsection{A truncated cone}
Let $A=\C\oplus M_2(\C)$. If we want to give a physical interpretation to this example, we can imagine that the state space describes the internal degrees of freedom of a couple of particles, one with spin $0$ and one with spin $1/2$.

Elements of $\mathbb{D}(A)$ are of the form
\begin{equation}\label{eq:3x3}
\begin{bmatrix}
1-t & 0 & 0 \\
0 & (t+x_3)/2 & (x_1-\mathrm{i}x_2)/2 \\
0 & (x_1+\mathrm{i}x_2)/2 & (t-x_3)/2
\end{bmatrix} \;,
\end{equation}
where $t\in [0,1]$, $x=(x_1,x_2,x_3)\in\R^3$ and $\|x\|\leq t$.
The state space is a truncated convex cone in $\R^4$, with typical section given by a closed $3$-ball (an illustration is in Figure \ref{fig:cone}a). The above matrix is a projection if and only if 
$$
\|x\|=t\in\{0,1\} \;.
$$
$\mathbb{P}(A)$ is the disjoint union, as expected, of a unit $2$-sphere and a point (the vertex).
Even if $\mathbb{D}(A)$ is homeomorphic to a closed unit ball in $\R^4$, it is clear from this example that the natural structure on the state space is that of a stratified space.
Notice that pure states form a measure zero subset of the boundary: boundary points with $0<\|x\|=t<1$ do not correspond to pure states. The maximally mixed state has $t=2/3$ and $x=0$.

The adjoint action of $U(A)=U(1)\times U(2)$ on a matrix \eqref{eq:3x3} doesn't change $t$, and rotates $x$ by the orthogonal matrix that is in the image of the projection $U(2)\to SO(3)$. There are two orbit types: points with $x=0$ and arbitrary $t$ (the ones lying on the symmetry axis of the cone) are of type $(U(1)\times U(2))$; all other points are of type $(U(1)\times U(1))$, and form a dense subset of the cone. Both orbit types have a non-empty boundary.

The matrix \eqref{eq:3x3} has eigenvalues $1-t$ and $(t\pm\|x\|)/2$. The rank is $1$ if{}f $\|x\|=t\in\{0,1\}$: this is the subset of the boundary of the cone made of pure states. The rank is $2$ if{}f $\|x\|<t=1$ or $0<\|x\|=t<1$: this is the set of boundary points that are not pure states (but can be written as a convex combination of exactly two pure states). The interior of the cone, points with $\|x\|<t$  and $0<t<1$, form the dense stratum of density matrices with maximal rank.

\begin{figure}[t]
\begingroup
\renewcommand{\arraystretch}{1.5}
\begin{tabular}{ccc}
\begin{tikzpicture}[scale=0.8,baseline=(current bounding box.center)]
	\draw[semithick] (1.5,4) arc (0:360:1.5 and 0.5);
	\draw[semithick] (-1.5,3.97) -- (0,0) -- (1.5,3.97);
	\draw (1.5,4) node[right] {$t=1$};
	\draw (0,0) node [right=2pt] {$t=0$};
  \filldraw (0,0) circle (0.03);

\end{tikzpicture}
&
\colorlet{mytriangle}{red!5}%green!5}
\colorlet{myedge}{red!80!black}%cyan!30!black}
\tdplotsetmaincoords{70}{70}
\begin{tikzpicture}[scale=3,baseline=(current bounding box.center),tdplot_main_coords]

\coordinate (O) at (0,0,0);
\coordinate (a) at (1,0,0);
\coordinate (b) at (0,1,0);
\coordinate (c) at (0.005,0.005,1);

\draw[->] (O) -- (1.55,0,0); %node[anchor=north west]{$a_1$};
\draw[->] (b) -- (0,1.25,0); %node[anchor=west]{$a_2$};
\draw[->] (O) -- (0,0,1.3); %node[anchor=south]{$a_3$};

\filldraw[thick,color=myedge,fill=mytriangle] (a) node[black,anchor=north]{$e_1$} -- (b) node[black] {\rule{0pt}{15pt}$\;\;\;e_2$} -- (c) node[black,anchor=east]{$e_3$} -- (a);

\draw[dashed] (0,0.28,0) -- (0,0.98,0);
\draw (O) -- (0,0.26,0);

\end{tikzpicture}
&
\tdplotsetmaincoords{155}{110}
\begin{tikzpicture}[tdplot_main_coords,baseline=(current bounding box.center)]

\coordinate[draw=none,shape=circle,fill, inner sep=1pt] (A) at (0,0,0);
\coordinate[draw=none,shape=circle,fill, inner sep=1pt] (B) at (3,3,0);
\coordinate[draw=none,shape=circle,fill, inner sep=1pt] (C) at (0,3,3);
\coordinate[draw=none,shape=circle,fill, inner sep=1pt] (D) at (3,0,3);

\draw[fill=green!30,opacity=0.2] (A.center) -- (D.center) -- (C.center) -- (A.center);
\draw[fill=red!30,opacity=0.2] (B.center) -- (D.center) -- (C.center) -- (B.center);

\path[semithick]
		(A) edge[dashed,gray] (B)
		(C) edge (D)
		(A) edge[color=green!50!black] (D)
		(B) edge (D)
		(B) edge[color=red!80!black] (C)
		(A) edge (C);

\end{tikzpicture}
\\
(a) State space of $\C\oplus M_2(\C)$.
&
(b) State space of $\C^3$.
&
(c) $\mathcal{S}(\C^2)\star\mathcal{S}(\C^2)$.
\end{tabular}
\endgroup
\vspace*{-5pt}
\caption{}\label{fig:cone}
\end{figure}

\subsection{The standard simplex}
Let $A$ be the commutative C*-algebra of $n\times n$ diagonal complex matrices, that we identify with $\C^n$.
In this case, the entries of a density matrix form a probability distribution on $n$ points (as expected: $A$ is isometrically $*$-isomorphic to the C*-algebra of functions on $n$ points). We can identify $\mathbb{D}(A)$ with the standard $(n-1)$-simplex in $\R^n$ and $\mathbb{P}(A)$ with the basis vectors of the canonical basis of $\R^n$ (the vertices of the simplex). An illustration is in Figure \ref{fig:cone}b. The maximally mixed state corresponds to the barycenter; pure states are a measure zero subset of the boundary.

Since $A$ is commutative, the adjoint action of $U(A)$ is trivial and there is a unique orbit type, the whole simplex.

On the other hand the stratification by rank consists of $n$ strata: the vertices (pure states) form the stratum of rank $1$, the interior is the stratum of rank $n$. For $1\leq k\leq n$, the stratum of rank $k$ consists of probability distributions with $k$ non-zero entries: this is the union of the interior of all $(k-1)$-faces of the simplex.

We can use this example to see Prop.~\ref{prop:11} at work. Let $A=A_1\oplus A_2$ with $A_1=A_2=\C^2$. The state space of both algebras is homeomorphic to a closed interval, say the green and red edges in Figure \ref{fig:cone}c. The end-points correspond to pure states (rank $1$), the interior of the coloured segments to mixed states (rank $2$). In the notations of Prop.~\ref{prop:11}, $X:=\mathcal{S}(A_1)=R_1\cup R_2$ and $Y:=\mathcal{S}(A_2)=S_1\cup S_2$ where $R_1$ are the vertices of the green edge, $R_2$ the interior, $S_1$ are the vertices of the red edge, $S_2$ the interior. The induced decomposition of the join has pieces:
\setlength{\leftmargini}{2em} %itemize/quote left margin
\begin{itemize}
\item $R_1,S_1$: the rank $1$ density matrices in $\C^4$;
\item $R_2,S_2$ (edges), $R_1\times S_2\times (0,1),R_2\times S_1\times (0,1)$ (each one union of two edges): their points correspond to density matrices of rank $2$ in $\C^4$;
\item $R_2\times S_2\times (0,1)$: the interior of the tetrahedron, giving density matrices of rank $3$.
\end{itemize}
As expected, the decomposition induced by the join gives the stratification into density matrices of fixed rank.

\section{On the local structure of density matrices}\label{sec:MnC}

Our aim in  this section is to investigate the structure of stratified space of $\mathcal{S}(M_n(\C))$. As explained above, it can be identified with the space of self-adjoint matrices with trace 1,  in any
complex vector space $E$   equipped with a positive definite Hermitian metric and such that $\dim E=n$:
$$\mathcal{S}(M_n(\C)) =\{f\in \Hb (E) \, \mid \,\, f\geq 0, \,\, \tr(f)=1 \}$$ 
Here $\Hb (E)$ denotes the subspace of  $\End $ containing Hermitian endomorphisms of $E$. It is very well  known (compare e.g.~with \cite[Sect.~1.7]{GMP}), that this space  is equipped with a structure of stratified space. We are interested in giving a closer look to its local structure. What we are going to do is to  provide not only an explicit description of all the strata but also of the conic structure around each point of every stratum. In other words, we carefully describe the structure of topological pseudo-manifold of $\mathcal{S}(M_n(\C))$.
In fact, it will turn out that our space satisfies much stronger conditions since the \emph{strata are smooth manifolds and the local homeomorphisms providing the conic structure are smooth in every stratum}.
Specifically, what we give here is
\begin{itemize}
\item a description of each stratum as \emph{the total space of a suitable smooth  bundle on an algebraic smooth variety} (compare with Proposition \ref{strata} and Remark \ref{rmkstrata}),
\item a description of a neighborhood of each stratum as a \emph{cone bundle over the  stratum, which is contained in the total space of a smooth vector bundle on an algebraic smooth variety} (compare with Theorem \ref{Main} and Remark \ref{main}). 	
\end{itemize}

Let us fix some standard notations.

\begin{notations}~
\label{not}
\begin{itemize}
\item For any Hermitian vector space $E$ we denote by $\Hb (E)$ (resp.~$\hp (E)$, $\hpd (E)$) the subspace  of $\End $ containing Hermitian (resp.~{positive} Hermitian, {positive definite} Hermitian) endomorphism of $E$. Observe that $\hp (E)$ ($\hpd (E)$) is closed (open) in~$\Hb(E)$.
\item For any subspace $X\subseteq \End$, we denote by $X_1$ the intersection of $X$ with the hyperplane of elements with trace 1 (so for instance, $\hpu$ stands for the space of positive Hermitian elements with trace 1). Observe that $\hpdu$ is closed in $\hpd$ and open in $\hpu$.
\item We set  	$\Sigma _i :=\{f\in \hpu(E) \, \mid \,\, \rk(f)\leq i \}$ and $\Sigma _i^0 :=\{f\in \hpu(E) \, \mid \,\, \rk(f)= i \}$.
\end{itemize}
\noindent More generally, we consider a Hermitian vector bundle $\ec \stackrel{\pi}{\rightarrow} M$   on a manifold $M$ and fix similar notations.
\begin{itemize}
\item  We denote by $\Hb (\ec)$ (resp.~$\hp (\ec)$, $\hpd (\ec)$) the subbundle of  $\Ends(\ec)$ containing Hermitian (resp.~{positive} Hermitian, {positive definite} Hermitian) endomorphism of $\ec$. By definition, we have $$ \Hb (\ec)_m=\Hb (\ec_m),\,\,\, \hp (\ec)_m=\hp (\ec_m), \,\,\, \hpd (\ec)_m=\hpd (\ec_m),\,\,  \,\,\, \forall m\in M.$$
\item Similarly as above, the lower index 1 stands for {elements with trace 1} and $\Sigma_i (\ec)$, and $\Sigma_i^0 (\ec)$, are defined in such a way that $$ \Sigma_i (\ec)_m=\Sigma_i (\ec_m),\,\,\, \Sigma_i^0 (\ec)_m=\Sigma_i^0 (\ec_m), \,\,\, \forall m\in M.$$
\end{itemize}
\end{notations}

Observe that in our new notations we have $\mathcal{S}(M_n(\C))=\hpu(E)$.
Our result will follow from a slightly more general one concerning the stratification of the space $\hpu(\ec)$.

Our first task is to provide each $\Sigma^0_i(\ec)$ with a structure of smooth manifold.
This will be accomplished in Lemma \ref{submanifold} and Proposition \ref{strata}.

\smallskip

 Denote by
$\Gi$ the \emph{Grassmann manifold of $(n-i)$-dimensional subspaces of} $E$. It is well known that $\Gi$ is a compact complex manifold of dimension $i (n-i)$, equipped with a short exact sequence of holomorphic vector bundles
\begin{equation}
\label{segrass}
0\rightarrow S_{\Gi}\rightarrow E_{\Gi}\rightarrow Q_{\Gi}\rightarrow 0,	
\end{equation}
where $S_{\Gi}$ denotes the \emph{tautological vector bundle}, 
$Q_{\Gi}$ denotes the \emph{universal quotient  bundle} and $E_{\Gi}$ denotes the \emph{trivial vector  bundle} on $\Gi$ (for any (n-i)-dimensional subspace $W\subseteq E$, we have $S_{\Gi, W}=W$). Of course, as a smooth  vector bundle on $\Gi$  can be identified with
$S_{\Gi}^{\perp}$ the orthogonal bundle of $S_{\Gi}$ in $E_{\Gi}$.

More generally, given a vector bundle on some manifold $\ec \stackrel{\pi}{\rightarrow} M$, we denote by $\Gi (\ec)$
the \emph{Grassmannian of subspaces of dimension (n-i) of the bundle}   $\ec $:
$$ \Gi (\ec)_m= \Gi (\ec_m)\, \,\,\,\, \forall m\in M.$$
The total space of $\Gi(\ec)$ is equipped with a short exact sequence of vector bundles as \eqref{segrass}.

\medskip
\begin{lemma}\label{submanifold}
The stratum 	$\Sigma _i^0(\ec)$ is a (locally closed) submanifold of $\Ends(\ec)$.	
\end{lemma}
\begin{proof}
As the statement is local, we are going to prove that 	$\Sigma _i^0$ is a  submanifold of $\Ends(E)$. First of all, we observe that $\hpdu$ is a locally closed submanifold of $\Ends(E)$, since it is open in a real affine subspace of $\Ends(E)$.
Furthermore, we recall that the subspace
$$\End _i:=\{f\in \End \mid \,\, \rk f=i \}\subseteq \End,  $$
is a smooth locally closed submanifold of $\End$. Indeed, by sending each $f\in \End$ to its graph, the space $\End$ can be identified with an open set in the Grassmannian $\G _n(E\oplus E)$. By means of this identification $\End _i$ becomes a Schubert
cell, which is known to be a smooth locally closed submanifold  of $\G_n( E\oplus E)$ (in fact it is  equipped with a structure of affine subspace \cite[Proposition 3.2.3]{Manivel}).

Consider $\mathrm{Fr}(E)$, the manifold of orthonormal frames of $E$ and the map
$$\psi: \mathrm{Fr}(E)\times M_n(\C)\rightarrow \End, $$
such that $\psi(B,A)$ is the operator defined by the matrix $A$ in the frame $B$. Via the map $\psi$,  $\mathrm{Fr}(E)\times M_n(\C)$ is a principal $U(n)$-bundle over $\End$. 

If we restrict the principal fibration $\psi$ to the submanifold 
$$\mathrm{Fr}(E)\times \pc \subset \mathrm{Fr}(E)\times M_n(\C),$$
\begin{equation}\label{dfnA}
A\in \pc \quad
\Leftrightarrow  \quad A=\left(\begin{matrix} 0 &  0 \\ 
0 & R\\ 
\end{matrix}\right), \,\,\,R\in M_i(\C),\,\,\, \rk R=i,
\end{equation}
then the restriction 
$$\Psi: \mathrm{Fr}(E)\times \pc \rightarrow \End _i$$
is surjective because any $f\in \End _i$ can be represented by a matrix
like in \eqref{dfnA} in some orthonormal frame. The morphism $\Psi $ is a $U(n-i) \times U (i)$-principal bundle over $\End _i$ (in fact it is a sub $(U (n-i) \times U (i))$-principal bundle of $\psi^{-1}(\End _i)$).

In order to prove the lemma it suffices to observe that $\Psi ^{-1}(\Sigma _i^0)$ is smooth and locally closed in $\mathrm{Fr}(E)\times \pc$.
This follows from our starting remark,
 since the subspace $\Psi ^{-1}(\Sigma _i^0)\subset \mathrm{Fr}(E)\times \pc$ is obtained by imposing
 $$R\in \hpdu (\C ^i)$$
 in the definition of $\pc$.
\end{proof}

\begin{prop}\label{strata}
The stratum 	$\Sigma _i^0(\ec)$ is diffeomorphic to the total space of $\hpdu (\sgpe)$.
\end{prop}
\begin{proof}
A point in the stratum $\Sigma _i^0(\ec)$ is represented by a pair 
$$(f_m,m)\mid\,\,\, f_m\in \hpu (\ec_m) \subset\Ends\ec _m \,\,\, \text{and} \,\,\, \dim \Ker f_m=n-i.$$
In other words, $\Ker f_m$ belongs to $\Gi(\ec_m)$ and the restriction of $f_m$ to the orthogonal of $\Ker f_m$ is Hermitian positive definite in $\Ends (\ec_m)$:
$$f_m\mid _{\Ker f_m ^{\perp}}\in \hpdu (\Ker f_m ^{\perp}).$$ 
 On the other hand, a point in the total space of $\Gi(\ec)$ is a pair $(W_m, m)$ such that $W_m\in \Gi(\ec_m)$, hence a point of the total space of the bundle $\hpdu (\sgpe)$ is represented by a triple
$$(g_m, W_m,m)\mid\,\,\, W_m\in \Gi (\ec_m) \,\,\, \text{and} \,\,\, g_m\in \hpdu(W _m^{\perp}).$$
Thus, the following map is well defined 
\begin{equation}\label{first}
\Sigma _i^0(\ec) \rightarrow  \hpdu (\sgpe), \,\,\,\, (f_m,m) \rightarrow (f_m\mid _{\Ker f_m ^{\perp}}, \Ker f_m, m).	
\end{equation}
Now we  define a map going the other way around. For any vector-subspace 	$V\subseteq E$ of a Hermitian  vector spce $E$, we denote by $\pi_V$ the \emph{orthogonal projection} over $V$:
$\pi_V: E\rightarrow V$. 

The following differentiable map is well defined
$$\hpdu (\sgpe) \rightarrow  \Sigma_i^0(\ec), \,\,\,\, (g_m, W_m, m)\rightarrow (g_m\circ \pi_{W_m^{\perp}},m),$$
and obviously the two maps  are inverse of each other.	

We conclude the proof by adding few words in order to convince the reader that the map \eqref{first} is indeed differentiable. Of course, the only doubt may concern the second component of \eqref{first}.
For any $g\in \End $, we denote by $\spec (g)$  the set of eigenvalues of $g$ and by
$$ R_g: \mathbb{C} \backslash \spec (g) \rightarrow \End,\,\,\,\, R_g(z):= \frac{1}{2\pi i(zI-g)}, $$
the \emph{resolvent function} of $g$.
Fix a neighborhood $U$ of a pair $(f_m, m)\in  \Sigma_i^0(\ec)$ and choose a neighborhood $U_m$ of $m$ such that $\ec \mid_{U_m}$ admits a trivialization $\ec \mid_{U_m}\cong U_m\times E$. Set
$$ a:= \min \{ \lambda \in \spec  f_m, \,\,\, \lambda \not=0 \}>0,$$ and fix $\epsilon$, s.t. $a>\epsilon >0$.
We can choose $U$ and $U_m$ in such a way that $U\subseteq V\times U_m$ for some open set $V$ of $\Ends (E)$. Thus, for any $(g_x,x)\in U$, we can view $g_x$ as belonging to $\Ends (E)$. Furthermore, we may assume $U$ small enough so as
$$\forall (g_x,x)\in U, \,\,\, 0\not=\lambda \in \spec g_x \,\,\, \Rightarrow \,\, \,\, \lambda >a-\epsilon.$$
Consider now a circle $\gamma\subset \mathbb{C}$ around $0$ and contained in a disk in $\mathbb{C}$, centered in 0 and with radius
 smaller than $a-\epsilon $. As long as $(g_x,x)$ moves in $U$, we have 
 $$\Im \left( \int _{\gamma } R _{g _x}(z) dz \right)=\Im \left( \int _{\gamma } \frac{1}{2\pi i(zI-g)} dz \right)= \Ker g_x    \in \Gi (E)$$
 (here $\Im(f)$ denotes the image of the operator $f$).
Finally, the second component of \ref{first} in $U$ is the map
$$U \rightarrow \Gi (E), \,\,\,\, (g_x,x) \rightarrow \Im \left(\int _{\gamma } R _{g_x}(z) dz \right),$$
and no doubt it is smooth.
\end{proof}

\begin{rem}\label{rmkstrata}
The bundle $\Hb _1(\sgpe)$ is a sub-vector bundle of $\Hb (\sgpe)$.  The fibres of the former are affine hyperplane in the fibres of the latter. Since $\hpdu(\sgpe)$ is open in $\Hb _1(\sgpe)$, Proposition \ref{strata} shows that 
$\Sigma _i^0(\ec)$ is diffeomorphic to a bundle on $\Gi (\ec)$ with fibres that are open sets in the vector bundle $\Hb _1(\sgpe)$. In particular, when $M$ is a point we have that $\Sigma _i^0$ is diffeomorphic to the total space of an open sub-bundle of $\Hb _1(\sgp)$,  that is to a smooth  bundle   on    an algebraic smooth variety.
\end{rem}

Now we come to the task of investigating the conic structure around the points of $\hpu(\ec)$. We start with some further notation.

As noted above,  $\Hb_1 (S_{\G_i(\ec)})$ is an \emph{affine hyperplane bundle} of $\Hb (S_{\G_i(\ec)})$. Observe that the fibres of $\Hb_1 (S_{\G_i(\ec)})$ do not contain the zero vector. Furthermore,  $\hpu (S_{\G_i(\ec)})$ is a subbundle of $\Hb_1 (S_{\G_i(\ec)})$, hence  its fibres do not contain the zero vector as well.

\begin{notations}
We denote by $\cc (S_{\G_i(\ec)})\subset \Hb (S_{\G_i(\ec)})$	the sub cone bundle of $\Hb (S_{\G_i(\ec)})$ whose fibres are the cones with vertex 0 and base the fibres of $\hpu (S_{\G_i(\ec)})$:
$$\cc (S_{\G_i(\ec)})_{(V_m, m)}=\cc(\hpu (S_{\G_i(\ec)})_{(V_m, m)})\subset \Hb (S_{\G_i(\ec)})_{(V_m, m)},\,\, \,\, V_m\in \Gi(\ec_m),$$
where $\cc(\hpu (S_{\G_i(\ec)})_{(V_m, m)})$ denotes the cone with vertex 0 and base  $\hpu (S_{\G_i(\ec)})_{(V_m, m)}$ in $\Hb (S_{\G_i(\ec)})_{(V_m, m)}$.
\end{notations}

\begin{thm}
\label{Main} There is a neighborhood of $\Sigma _i^0(\ec)$ in $\hpu(\ec)$ which is isomorphic, as a stratified space, to 
a neighborhood of $\hpdu (\sgpe)$ inside the total space of $\cc (S_{\G_i(\ec)})\oplus \hpdu (\sgpe)\subset \Hb (S_{\G_i(\ec)})\oplus \hpdu (\sgpe)$.
\end{thm}
\begin{proof}
Since the argument is purely local in the manifold $M$, in what follows we drop the reference to the point $m$ varying in $M$.
We argue by induction on the dimension of our Hermitian vector space $E$ and continue to use the notations introduced in Proposition \ref{strata}. By induction, we may assume that 	the fibres of $\cc(S_{\G_i})\oplus \hpdu (\sgp)$ are stratified spaces with lower dimension. Fix a Hermitian $f\in \Sigma_i^0$. We are going to construct an isomorphism of stratified spaces between a suitable neighborhood of $f$ in $\hpu(E)$ with a neighborhood of $(f\mid _{\Ker f^{\perp}}, \Ker f)$ in the total space of the fibre bundle
$\cc (S_{\G_i})\oplus \hpdu (\sgp)$.
Set
$$ a:= \min \{ \lambda \in \spec  f, \,\,\, \lambda \not=0 \}>0,$$ and fix $a>\epsilon >0$.
We are allowed to choose a neighborhood $U$ of $f$ in $ \hpu(E)$  in such a way any Hermitian $g\in U$ satisfies
$$\lambda \in \spec g \,\,\, \Rightarrow \,\,\, 0\leq \lambda <\epsilon \,\, \vee \,\, \lambda >a-\epsilon.$$
Consider now a circle $\gamma_1\subset \mathbb{C}$ around $0$ and a closed path $$\gamma _2\subset \{z\in \mathbb{C} \mid \,\,\, Rez>a-\epsilon \}$$
 and surrounding $\spec f$. 
We define a differentiable map
$$F: U \rightarrow \left( \cc(\Sc(S_{\G_i}))\oplus \hpdu (\sgp) \right),$$
$$F(g):= (\int _{\gamma_1} zR_g(z) dz, (1-\alpha)^{-1}\int _{\gamma_2} zR_g(z) dz), \,\,\,\, \alpha:= \tr \int _{\gamma_1} zR_g(z) dz. $$
We have
$$\int _{\gamma_1}zR_g(z) dz\in \hp(S_{\G_i, W}), \,\,\, (1-\alpha)^{-1}\int _{\gamma_2} zR_g(z) dz\in \hpdu (\sgp)_W,$$
where $W:= \Im (\int _{\gamma_1}R_g(z) dz).$
Furthermore, we have $$\alpha ^{-1}\int _{\gamma_1}zR_g(z) dz\in \hpu(S_{\G_i, W}), \,\,\, \text{if} \,\,\, \alpha\not=0.$$
By inspection, we see that the morphism above is differentiable in each stratum.
An explicit inverse of $F$ is provided by
$$(h, k)\in (\Ends (W), \Ends (W^{\perp}))\rightarrow h\circ \pi_W + (1-\tr h)k\circ \pi_{W^{\perp}}, $$
as soon as $\tr(h)\ll 1$.
\end{proof}

\begin{rem}\label{main}
The description given in Theorem \ref{Main} shows that a neighborhood of $\Sigma_i^0(\ec)$ is isomorphic to 
$\cc (S_{\G_i(\ec)})\oplus \hpdu (\sgpe)$ as a stratified space. So, the conic structure around any point of $\Sigma_i^0(\ec)$ is that of a cone bundle over $\Sigma_i^0(\ec)$. In particular, when $M$ is a point we have that a neighborhood of $\Sigma _i^0$ is isomorphic to 
$\cc (S_{\G_i})\oplus \hpdu (\sgp)$, that is a cone bundle  on an algebraic smooth variety.
\end{rem}

\begin{cor}~\label{last}
\begin{itemize}\itemsep=2pt
\item The stratum 	$\Sigma _i^0$ is diffeomorphic to the total space of $\hpdu (\sgp)$.
\item There is a neighborhood of $\Sigma _i^0$ in $\hpu(E)$ which is isomorphic, as a stratified space, to 
a neighborhood of $\hpdu (\sgp)$ inside the total space of $\cc (S_{\G_i})\oplus \hpdu (\sgp)\subset \Hb (S_{\G_i})\oplus \hpdu (\sgp)$.
\end{itemize}	
\end{cor}

\section{Closing remarks}\label{sec:infdim}

In this paper we discussed the structure of Whitney stratified space of the state space of a finite-dimensional C*-algebra. A natural question is in what form these results can be generalized to the infinite-dimensional case. While a detailed study is postponed to future works, we collect here a few remarks and point to possible directions of research.

\smallskip

If $A=\mathcal{K}(H)$ is the C*-algebra of compact operators on a separable Hilbert space $H$, the state space can be identified with the subspace of $\mathcal{L}^1(H)$ of positive operators with trace $1$, the so-called \emph{density operators}. In this case, it should be possible to carry on a study similar to the one in Section \ref{sec:MnC}.
Pure states are still given by rank $1$ projections.
The set of projections in a Banach $*$-algebra is an analytic Banach submanifold \cite[Prop.~1.1]{AM}. Identifying rank $1$ projections in $\mathcal{L}^1(H)$ with rays in $H$ one has a realization of $\mathbb{P}(H)$ as an analytic Banach submanifold of $\mathcal{L}^1(H)$.
More generally, the orbit of a density operator under the adjoint action of the Banach-Lie group of unitary endomorphisms of $H$ is an analytic Banach manifold \cite{CIJM}.

For an arbitrary unital C*-algebra $A$, states can be represented by density operators of the enveloping von Neumann algebra.
Let $\pi:A\to\mathcal{B}(H)$ be the universal representation of $A$ (the direct sum of GNS representations of all states of $A$) and $\pi(A)''$ the enveloping von Neumann algebra of $A$ (the double commutant of $\pi(A)$).
Every state of $A$ has a unique extension to a normal state of $\pi(A)''$, i.e.~a state given by $a\mapsto \tr(\rho\hspace{1pt}a)$ for some density operator $\rho\in\mathcal{L}^1(H)$. See \cite[\S3.7--3.8]{Ped}.

\smallskip

One possible direction could be then to look at a generalization of Whitney stratified spaces where smooth manifolds are replaced by analytic Banach manifolds.
A drawback of this approach is that we cannot expect to find a ``nice'' stratification of $\mathcal{S}(A)$, in general, if the underlying topology is the weak-* topology. For example, if there is a stratification of $\mathcal{S}(A)$ whose underlying topology is the weak-* topology and the pure state space $\overline{\mathcal{P}(A)}$ is a stratum, necessarily such a stratum must be a finite-dimensional manifold. This follows from the observation that $\overline{\mathcal{P}(A)}$ is compact and Hausdorff in the weak-* topology, and from the following easy lemma.

\begin{lemma}
Any compact Hausdorff Banach manifold is finite-dimensional.
\end{lemma}

\begin{proof}
Let $X$ be a compact Hausdorff topological space and assume, by contraddiction, that there exists a homeomorphism $\Phi:U\to\Phi(U)$ from an open non-empty subset $U$ of $X$ and an open non-empty subset of some Banach space $V$, equipped with norm topology. Now, $\Phi(U)$ contains an open ball, and then a closed ball $\overline{B}$ of smaller radius.
Since $\overline{B}$ is homeomorphic, via $\Phi^{-1}$, to a closed (hence compact) subset of $X$, it is it-self compact.
On the other hand, an infinite-dimensional normed vector space has no non-empty open subset with compact closure.
\end{proof}

It is worth mentioning that smooth Banach manifolds (but also stratified spaces) embed fully faithfully into the category of diffeological spaces \cite{IZ}, and on any subset of a diffeological space there is a natural subspace diffeology. Thus, $\mathcal{S}(A)$ has a natural diffeology induced by the Banach diffeology of $A^*$ \cite[Ex.~72]{IZ}. The underlying topology, however, is still the norm topology.

\end{document}